\documentclass{amsart}

\usepackage[dvipdfmx]{graphicx}

\newtheorem{theorem}{Theorem}[section]
\newtheorem{lemma}[theorem]{Lemma}
\newtheorem{corollary}[theorem]{Corollary}%
\newtheorem{proposition}[theorem]{Proposition}

\theoremstyle{definition}

\theoremstyle{remark}
\newtheorem{remark}[theorem]{Remark}

\numberwithin{equation}{section}

\begin{document}

\markboth{Masakazu Teragaito}
{Generalized torsion elements and hyperbolic links}

%%%%%%%%%%%%%%%%%%%%% Publisher's Area please ignore %%%%%%%%%%%%%%
%\catchline{}{}{}{}{}
%%%%%%%%%%%%%%%%%%%%%%%%%%%%%%%%%%%%%%%%%%%%%%%%%%%%%%%%%%%%%%%%%%%

\title{Generalized torsion elements and hyperbolic links}

\author{Masakazu Teragaito
%\footnote{Typeset names in 8~pt Times Roman, uppercase
%and lightface.  Use footnotes only to indicate if permanent and
%present addresses are different. Funding information should go
%in the Acknowledgement section.}
}

\address{Department of Mathematic Education, Hiroshima University,
1-1-1 Kagamiyama, Higashi-hiroshima, 7398524, Japan}
\email{teragai@hiroshima-u.ac.jp}
\thanks{This work was supported by
JSPS KAKENHI Grant Number 20K03587.}

%\footnote{
%affiliation and mailing addresses in 8pt Times italic.} \\
%}

\maketitle

\begin{abstract}
In a group,
a generalized torsion element is a non-identity element whose some non-empty finite product of its conjugates yields
the identity.  Such an element is an obstruction for a group to be bi-orderable.
We show that the Weeks manifold, the figure-eight sister manifold, and the complement of Whitehead sister link
admit generalized torsion elements in their fundamental groups.
In particular, the Whitehead sister link, which is the pretzel link of type $(-2,3,8)$,  can be generalized to hyperbolic pretzel links of type $(-2,3,2n)\ (n\ge 4)$.
These give the first examples of hyperbolic links whose link groups admit generalized torsion elements.
\end{abstract}

%\keywords{generalized torsion element, bi-ordering, hyperbolic link, figure-eight sister manifold, Whitehead sister link, pretzel link}

%\subjclass[2010]{Primary 57M25; Secondary 06F15, 20F60}
\maketitle

\section{Introduction}

In a group $G$, a non-trivial element $g$ is called a \textit{generalized torsion element\/} if
some non-empty finite   product of its conjugates is equal to the identity.
That is, $g^{a_1}g^{a_2}\dots g^{a_k}=1$ for some $a_1,a_2,\dots, a_k\in G$.
Here $g^a$ denotes a conjugate of $g$ by $a\in G$.

Every knot or link group is torsion free.
However, it may contain a generalized torsion element.
For example, all torus knot groups satisfy it.
Naylor and Rolfsen \cite{NR} gave the first example of hyperbolic knot, which is the $(-2)$-twist knot,
whose knot group contains a generalized torsion element.
Then we showed that any negative twist knot enjoys the same property \cite{T}.

The existence of generalized torsion element is an obstruction for a group $G$ to  be bi-orderable.
Recall that $G$ is said to be \textit{bi-orderable\,} if it admits a strict total ordering which is invariant
under multiplication from left and right sides;
if $a<b$, then $gah<gbh$ for any $g$ and $h$ in $G$.
If $G$ is bi-orderable, then any non-identity element $g$ is bigger or smaller than the identity.
Let $g>1$.
Since conjugation preserves the order, $g^a>1$ for any $a\in G$.
Then any product of such elements is still bigger than $1$.
Similarly for the case  $g<1$.
Thus any bi-orderable group has no generalized torsion element.

Conversely, even if $G$ has no generalized torsion element, 
we cannot claim that $G$ is bi-orderable \cite{B,BL,MR}.
However, we expect that such phenomenon does not occur among $3$-manifold groups \cite{MT}.
Many $3$-manifold groups are know to be not bi-orderable.
For, any finitely generated bi-orderable group surjects on the infinite cyclic group (see \cite{CR}).
In particular, if a $3$-manifold $M$ has finite first homology group, then $\pi_1(M)$ is not  bi-orderable.
In \cite{MT}, we propose a conjecture that 
if $\pi_1(M)$  is not bi-orderable for a $3$-manifold $M$, then it contains
a generalized torsion element.
This is solved affirmatively  for Seifert fibered manifolds, Solvable manifolds, and 
a few infinite families of hyperbolic manifolds.

In this paper, we first examine small volume cusped hyperbolic $3$-manifolds.
The Weeks manifold is the unique closed orientable hyperbolic $3$-manifold of smallest volume \cite{GMM}.
By \cite{CM}, the minimal volume orientable hyperbolic $3$-manifold with one cusp is homeomorphic
to either the figure-eight knot complement or the figure-eight sister manifold, which is also called the sibling manifold.
Furthermore, 
 the minimal volume orientable hyperbolic $3$-manifold with two cusps is homeomorphic
to either the Whitehead link complement or the Whitehead sister manifold, which is the $(-2,3,8)$-pretzel link
complement.
The figure-eight knot complement and the Whitehead link complement have
bi-orderable fundamental groups \cite{PR, KR},
so they do not admit a generalized torsion element in their fundamental groups.

Our first result claims that the others admit generalized torsion elements in the fundamental groups.

\begin{theorem}\label{thm:main1}
The Weeks manifold, the figure-eight sister manifold and the Whitehead sister manifold 
admit generalized torsion elements in their fundamental groups.
\end{theorem}

As remarked above, the Whitehead sister manifold is the complement of the $(-2,3,8)$-pretzel link.
Second, we generalize this to an infinite family of pretzel links.

\begin{theorem}\label{thm:main2}
Let $L$ be the two-component pretzel link $P(-2,3,2n)$.
If $n\ge 1$, then the link group $\pi_1(S^3-L)$ contains a generalized torsion element.
\end{theorem}

This theorem immediately gives the first examples of hyperbolic links whose
link groups admit generalized torsion elements.

\begin{corollary}
There are infinitely many hyperbolic two-component links whose link groups
admit generalized torsion elements. 
\end{corollary}

\begin{proof}
If $n\ge 4$, then $P(-2,3,2n)$ is hyperbolic \cite{MP}.
The conclusion immediately follows from Theorem \ref{thm:main2}.
\end{proof}

Throughout the paper, 
we use the notation $\bar{g}=g^{-1}$, $[g,h]=g^{-1}h^{-1}gh$ and $g^a=a^{-1}ga$
for the inverse, the commutator and the conjugate in a group.

%%%

\section{The Weeks manifold and the figure-eight sister manifold}

We start from the exterior $W$ of the Whitehead link.
Then $\pi_1(W)$ has the following Wirtinger presentation
\[
\pi_1(W)=\langle a, b \mid aba\bar{b}\bar{a}bab=bab\bar{a}\bar{b}aba\rangle,
\]
where $a$ and $b$ are meridians of the components as shown in Fig.~\ref{fig:link}.

\begin{figure}[htpb]
%\begin{center}%\includegraphics*[bb=0 0 182 192]{link.pdf}
\includegraphics*[width=4cm]{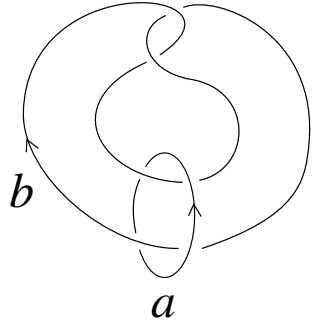}
\caption{The Whitehead link and the meridians $a$, $b$}\label{fig:link}
%\end{center}
\end{figure}

The figure-eight sister manifold is described as the resulting manifold by performing $5/1$-Dehn filling
on one boundary component of $W$ (see \cite{G}).
Two components of the Whitehead link are interchangeable by an isotopy,
so there is no ambiguity for the choice of boundary component.

\begin{theorem}\label{thm:w}
For a slope $m/n\ (n\ge 1)$, let $W(m/n)$ be the resulting manifold by $m/n$-Dehn filling on one boundary component of 
the Whitehead link exterior $W$.
If $m\ge 2n$, then $\pi_1(W(m/n))$ contains a generalized torsion element.
In particular, the figure-eight sister manifold $W(5)$ satisfies this.
\end{theorem}

\begin{proof}
Perform $m/n$-surgery along the component with meridian $a$.
Then the surgery yields a relation
\begin{equation}(\bar{b}^{ab\bar{a}\bar{b}}b)^n a^m=1,
\end{equation}
because the longitude is $\bar{b}^{ab\bar{a}\bar{b}}b$.
This gives
\begin{equation}\label{eq1}
(ba\bar{b}\bar{a}\bar{b}ab\bar{a})^na^m=1.
\end{equation}
Hence
\begin{equation}
\pi_1(W(m/n))=\langle a,b \mid 
aba\bar{b}\bar{a}bab=bab\bar{a}\bar{b}aba, 
(ba\bar{b}\bar{a}\bar{b}ab\bar{a})^na^m=1\rangle.
\end{equation}

The first relation gives
\begin{equation}\label{eq2}
ba\bar{b}\bar{a}\bar{b}ab=a^{\bar{b}\bar{a}}a^b\bar{a}.
\end{equation}
Hence (\ref{eq1}) and (\ref{eq2}) yield
$(a^{\bar{b}\bar{a}}a^b\bar{a}^2)^na^m=1$,
so  we have
\begin{equation}\label{eq3}
UU^{a^2}U^{a^4}\cdots U^{a^{2(n-1)}}a^{m-2n}=1,
\end{equation}
where $U=a^{\bar{b}\bar{a}}a^b$.
Thus if $m\ge 2n$, then the left hand side of (\ref{eq3}) is a product of conjugates of only $a$.

Since $H_1(W(m/n))=\mathbb{Z}\oplus \mathbb{Z}_{|m|}$ and the element $a$ goes to a generator of $\mathbb{Z}_{|m|}$ summand,
$a$ is non-trivial in $\pi_1(W(m/n))$.
Hence if $m\ge 2n$, then (\ref{eq3}) shows that the element $a$ is a generalized torsion element in $\pi_1(W(m/n))$.
\end{proof}

\begin{remark}
\begin{enumerate}
\item
Since $W(1)$ is the exterior of the trefoil, $\pi_1(W(1))$ contains a generalized torsion element (see  \cite{NR}).
\item
If $m=-1$, then $W(-1/n)$ gives the exterior of the $n$-twist knot.
For example, $W(-1)$ is the figure-eight knot exterior.
Since the knot group of any positive twist knot is known to be bi-orderable \cite{C}, $\pi_1(W(-1/n))$ does not contain a generalized torsion element for any $n\ge 1$.
On the other hand, if $m=1$, then $\pi_1(W(1/n))$ contains a generalized torsion element, because the knot group
of any negative twist knot admits a generalized torsion element \cite{T}.
\end{enumerate}

\end{remark}

Let $M=W(m/n)$, and let $M(r)$ denote the $r$-Dehn filling on $M$.
The Weeks manifold is $M(5/2)$ for $M=W(5)$ (see \cite{CD}).

\begin{corollary}\label{cor:weeks}
If  $m\ge 2n$.
$\pi_1(M(r))$ contains a generalized torsion element for any slope $r\in \mathbb{Q}\cup \{1/0\}$.
In particular, the Weeks manifold $M(5/2)$ satisfies this.
\end{corollary}

\begin{proof}
Since $\pi_1(M(r))$ is the quotient of $\pi_1(W(m/n))$,
the relation (\ref{eq3}) still holds in $\pi_1(M(r))$.
Set $r=p/q$. 
The element $a$ projects to a generator of $\mathbb{Z}_{|m|}$ summand of $H_1(M(r)))=\mathbb{Z}_{|p|}\oplus \mathbb{Z}_{|m|}$, so
it is nontrivial in $\pi_1(M(r))$.
Hence $a$ remains to be a generalized torsion element in $\pi_1(M(r))$.
\end{proof}

%%%%%%
\section{The Whitehead sister link and pretzel links}

\subsection{Tunnel number one pretzel links}

Let $L$ be the pretzel link $P(-2,3,2n)$ for $n\ne 0$.
In particular, $P(-2,3,8)$ is the Whitehead sister link.
Let $G=\pi_1(S^3-L)$ be the link group.
The purpose of this subsection is to obtain a presentation of $G$ 
with two generators and a single relation based on the fact that $L$ has tunnel number one.

Figure \ref{fig:tunnel} shows the unknotting tunnel $\gamma$ for $L$.
This means that the exterior of the regular neighborhood $N$ of $L\cup \gamma$
is a genus two handlebody $H$.
Hence the exterior of $L$ is obtained from $H$ by attaching a $2$-handle 
along the co-core  $\ell$ of the regular neighborhood of $\gamma$, which is regarded as a $1$-handle attached on the regular neighborhood of $L$.
This implies that the link group $G$ has two generators and a single relation which comes from
the $2$-handle addition.

\begin{figure}[htpb]
%\begin{center}%\includegraphics*[bb=0 0 182 192]{link.pdf}
\includegraphics*[width=10cm]{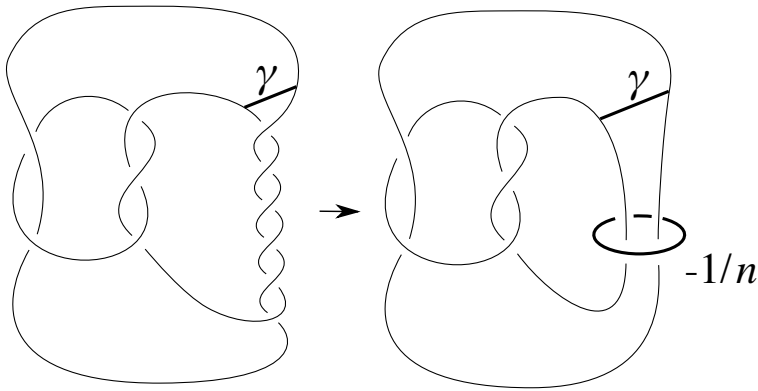}
\caption{The pretzel link $P(-2,3,2n)$, where $n=3$, and its unknotting tunnel $\gamma$}\label{fig:tunnel}
%\end{center}
\end{figure}

To get  a rank two presentation of $G$,
we will trace the co-core $\ell$ on $\partial N$ during
the unknotting transformation of $N$.

First, we replace  the $2n$-twists on $L$  with $(-1/n)$-surgery on the additional unknotted circle
as shown in Fig.~\ref{fig:tunnel}.
Then Fig.~\ref{fig:start} shows $N$ and the co-core $\ell$.

\begin{figure}[htpb]
%\begin{center}%\includegraphics*[bb=0 0 182 192]{link.pdf}
\includegraphics*[width=6cm]{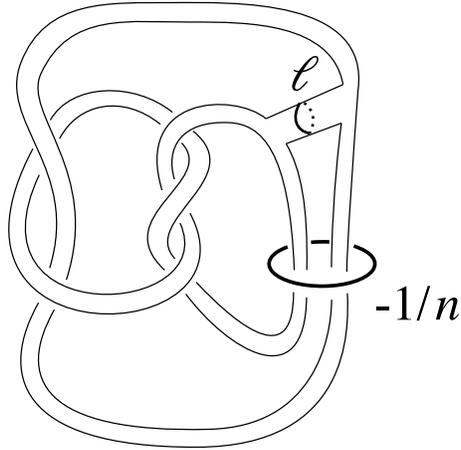}
\caption{The regular neightborhood $N$ of $L\cup \gamma$ and the co-core $\ell$ on $\partial N$}\label{fig:start}
%\end{center}
\end{figure}

As illustrated  in Fig.~\ref{fig:h1}, \ref{fig:h2} and \ref{fig:h3},
we transform $N$.
Here, the loop $\ell$ in Fig.~\ref{fig:h2} and \ref{fig:h3} is described as
a band sum of two circles for simplicity.

\begin{figure}[htpb]
%\begin{center}%\includegraphics*[bb=0 0 182 192]{link.pdf}
\includegraphics*[width=7cm]{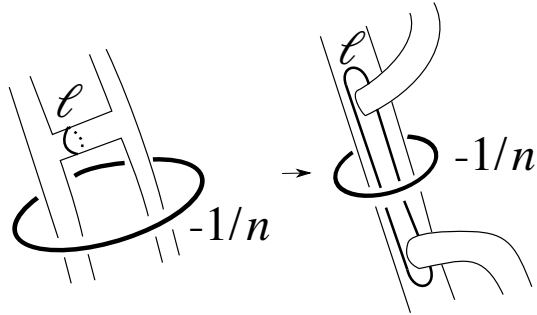}
\caption{The first move of $N$}\label{fig:h1}
%\end{center}
\end{figure}

\begin{figure}[htpb]
%\begin{center}%\includegraphics*[bb=0 0 182 192]{link.pdf}
\includegraphics*[width=10cm]{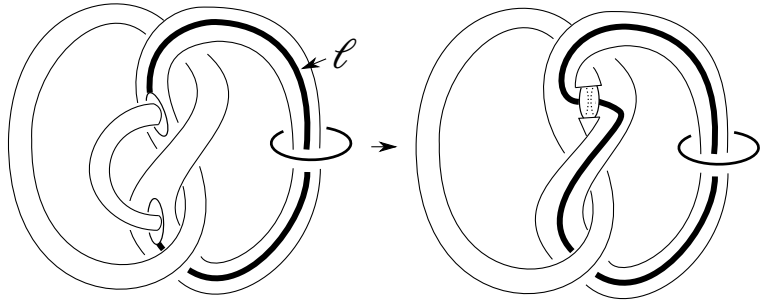}
\caption{The second move of $N$. Here, the loop $\ell$ is described as a band sum of two circles.}\label{fig:h2}
%\end{center}
\end{figure}

\begin{figure}[htpb]
%\begin{center}%\includegraphics*[bb=0 0 182 192]{link.pdf}
\includegraphics*[width=8cm]{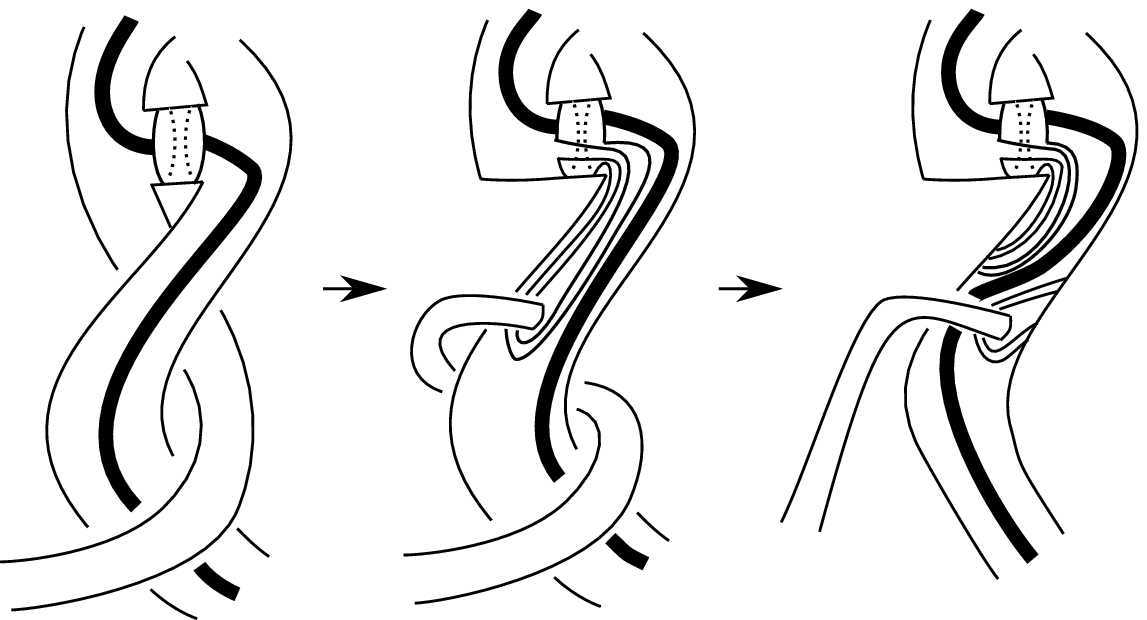}
\caption{The third move of $N$}\label{fig:h3}
%\end{center}
\end{figure}

Figure \ref{fig:h4} shows the final form of $N$ with $\ell$ on $\partial N$.
It is easy to see that the outside of $N$ is also a genus two handlebody $H$.
Here, the loops $\alpha$ and $\beta$ bound mutually disjoint non-separating meridian disks of $H$.
Hence if we take the generators $a$ and $b$ of $\pi_1(H)$ as the duals of $\alpha$ and $\beta$,
then we can easily express $\ell$ as a word of $a$ and $b$ in $\pi_1(H)$ by following
the intersection points between $\ell$ and $\alpha$ and $\beta$.

\begin{figure}[htpb]
\begin{center}%\includegraphics*[bb=0 0 182 192]{link.pdf}
\includegraphics*[width=10cm]{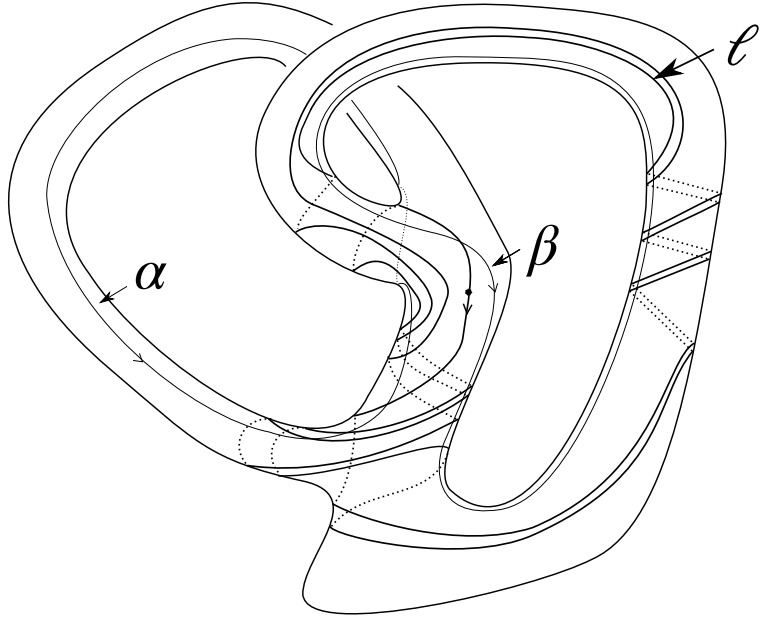}
\caption{The outsider of $N$ is a genus two handlebody $H$.
The loops $\alpha$ and $\beta$ bound disjoint meridian disks of $H$.
Here, $n=3$.}\label{fig:h4}
\end{center}
\end{figure}

\begin{proposition}\label{prop:p}
The link group $G$ has the presentation
\begin{equation}\label{eq:p}
G=\langle a, b \mid 
a\bar{b}^{n-1}a\bar{b}ab=ba\bar{b}a\bar{b}^{n-1}a\rangle.
\end{equation}
\end{proposition}

\begin{proof}
The exterior of $L$ is obtained from the genus two handlebody $H$ by attaching a $2$-handle
along the loop $\ell$.
Let $a$ and $b$ be the generators of $\pi_1(H)$, which are duals of $\alpha$ and $\beta$ as shown in Fig.~\ref{fig:h4}.
Hence the link group $G$ has a presentation with generators $a$ and $b$ and a single relation coming from $\ell$.
We orient $\alpha$, $\beta$ and $\ell$ as in Fig.~\ref{fig:h4},
and follow $\ell$ from the dot.
Then we get the relation as in the statement.
\end{proof}

\subsection{Generalized torsion elements}

\begin{lemma}\label{lem:rel}
In $G$, $[a,b\bar{a}b^{n-1}\bar{a}b]=1$.
\end{lemma}

\begin{proof}
The relation of (\ref{eq:p}) gives
$U a^b = a^{\bar{b}} U$,
where $U=a \bar{b} ^{n-1} a$.
Hence we have $U^b a^{b^2}=aU^b$, so $a=a^{b^2 \bar{U}^b}$.
This gives $[a,b^2\bar{U}^b]=1$, which yields the conclusion.
\end{proof}

\begin{lemma}\label{lem:decomp}
Let $w(\bar{a},b)$ be a word
containing only $\bar{a}$ and $b$. 
Then the commutator $[a,w(\bar{a},b)]$
can be expressed as a product of conjugates of the commutator $[a,b]$.
\end{lemma}

\begin{proof}
In general, we have an equation $[a,uv]=[a,v][a,u]^v$.
Since $[a,\bar{a}]=1$,
$[a,w(\bar{a},b)]$ is decomposed into a product of conjugates of only $[a,b]$.
\end{proof}

\begin{theorem}\label{thm:pretzel}
If $n\ge 1$, then $G$ admits a generalized torsion element.
\end{theorem}

\begin{proof}
By Lemma \ref{lem:decomp},
$[a,b\bar{a}b^{n-1}\bar{a}b]$ is expressed as a product of conjugates of $[a,b]$ if $n\ge 1$. 
We know that $[a,b]\ne 1$ in $G$, because $G$ is not abelian.
(The only knots and links  whose  groups are abelian are the unknot and the Hopf link.)
Then Lemma \ref{lem:rel} implies that $[a,b]$ is a generalized torsion element in $G$.
\end{proof}

\begin{remark}
When $n=0$, then link $L$ is the connected sum of the trefoil and the Hopf link.
Since the knot group of the trefoil contains a generalized torsion element, so does $G$.
For $n<0$, our argument in the proof of Theorem \ref{thm:pretzel}  does not work.
\end{remark}

\noindent
\textbf{Proof of Theorem \ref{thm:main1}.}
This follows from Theorems \ref{thm:w} and \ref{thm:pretzel}  and Corollary \ref{cor:weeks}. 
\hfill $\Box$

\medskip
\noindent
\textbf{Proof of Theorem \ref{thm:main2}.}
This immediately follows from Theorem \ref{thm:pretzel}. \hfill $\Box$

\subsection{Problems}

Kin and Rolfsen \cite{KR} show that
the pretzel links $P(-2,2k+1,2n)$ and $P(-2,-2k+1,2n)$ for $k\ge 1$ and $n\ge 2$
do not have bi-orderable link groups.
Hence their link groups are expected to admit generalized torsion elements
beyond our Theorem \ref{thm:main2}.
Fortunately, these pretzel links still have tunnel number one.
Thus it is possible to apply our procedure to get a presentation of link group with two generators and
a single relation.
We tried this,  but we could not find a generalized torsion element, 
because of  the complicated relation.

The exterior of pretzel links $P(-2,3,2n)$ is obtained from the magic manifold
by suitable Dehn filling.
If the magic manifold contains a generalized torsion element in its fundamental group,
then there would be a chance for the element to give a generalized torsion element
for $P(-2,3,2n)$ as in the proof of Corollary \ref{cor:weeks}.
However, the result of \cite{KR} seems to  suggest that 
the magic manifold would have bi-orderable fundamental group, so 
there is no generalized torsion element.

%%%%

\end{document}